\documentclass{amsart}
\usepackage[T1]{fontenc}
\usepackage[english]{babel}
\usepackage{amsmath,amsthm,amssymb, latexsym,geometry,fancyhdr,stackrel,enumerate}
\usepackage[all]{xy}
\usepackage{srcltx}
\usepackage{hyperref}

\date{\today}
\swapnumbers
\theoremstyle{plain}
\newtheorem{thm}{Theorem}[section]
\newtheorem{lem}[thm]{Lemma}
\newtheorem{prop}[thm]{Proposition}
\newtheorem{cor}[thm]{Corollary}
\newtheorem*{Thm}{Theorem}
\theoremstyle{definition}

\newtheorem{rems}[thm]{Remarks}
\newtheorem{ex}[thm]{Example}

\newcommand{\ts}[1]{\normalfont{\textsf{#1}}}
\newcommand{\K}{\ts k}
\newcommand{\kQ}{\K Q}

\newcommand{\ol}[1]{\overline{#1}}
\renewcommand{\a}{\alpha}
\renewcommand{\b}{\beta}
\renewcommand{\c}{\gamma}
\newcommand{\e}{\varepsilon}
\renewcommand{\l}{\lambda}
\renewcommand{\d}{ \delta}
\newcommand{\s}{\sigma}
\newcommand{\f}{\varphi}
\newcommand{\G}{\Gamma}

\newcommand{\R}{\mathcal{R}}

\newcommand{\mor}[3]{$#1\colon #2 \to #3$}
\newcommand{\arete}[3]{$#1\colon #2$ --- $ #3$}

\newcommand{\rad}[1]{\ts{rad} #1}
\renewcommand{\ker}[1]{\ts{Ker}\  #1}
\newcommand{\Hom}[3]{\ts{Hom}_{#1}(#2,\,#3)}

\newcommand{\HH}[2]{\ts{HH}^{#1}(#2)} 

\newcommand{\SH}[1]{\ts H_{1}(#1)}

\title[Special biserial algebras with $\ts{HH}^1(A)=0$.]{Special biserial algebras with no outer derivations}

\author[I. Assem]{Ibrahim Assem}
\address{D\'epartement de Math\'ematiques, Universit\'e
  de Sherbrooke, Sherbrooke, Qu\'ebec, Canada J1K 2R1}
\email{ibrahim.assem@USherbrooke.ca}

\author[J.C. Bustamante]{Juan Carlos Bustamante}
\address{Departamento de Matem\'aticas, Universidad
San Fransisco de Quito. Cumbay\'a - Quito,
  Ecuador}
\curraddr{D\'epartement de Math\'ematiques, Universit\'e
  de Sherbrooke, Sherbrooke, Qu\'ebec, Canada J1K~2R1}
\email{juan.carlos.bustamante@usherbrooke.ca}

\author[P. Le Meur]{Patrick Le Meur}
\address{Laboratoire de Math\'ematiques - Universit\'e Blaise Pascal \& CNRS - Campus Scientifique des C\'ezeaux - BP 80026 - 63171 Aubi\`ere cedex - France}
\email{patrick.lemeur@math.univ-bpclermont.fr}


\begin{document}

\begin{abstract}
Let $A$ be a special biserial algebra over an algebraically closed field. We show that the first Hohchshild cohomology group of $A$ with coefficients in the bimodule $A$ vanishes if and only if  $A$ is representation finite and simply connected (in the sense of Bongartz and Gabriel), if and only if the Euler characteristic of $Q$ equals the number of indecomposable non uniserial projective injective $A$-modules (up to isomorphism). Moreover, if this is the case, then all the higher Hochschild cohomology groups of $A$ vanish.
\end{abstract}

\maketitle

\section*{Introduction} Let $\K$ be an algebraically closed field and $A$ a finite dimensional $\K-$algebra. It is a reasonable question to try to relate the Hochschild cohomology groups of $A$ with the representation theory of $A$, that is, with properties of the category $\ts{mod}\ A$ of finitely generated right $A-$modules.

We are here interested in the vanishing of the first Hochschild cohomology group $\HH{1}{A} $ of $A$ with coefficients in the bimodule $_{A}A_{A}.$ In \cite{Skow92}, Skowro\'nski asked for which triangular algebras $A$ do we have $\HH{1}{A}=0$ if and only if $A$ has no proper Galois covering.

Since then, this problem has been investigated by several authors due to its connection with the notion of simple connectedness. In \cite{BG82}, Bongartz and Gabriel define a representation finite algebra to be simply connected if its Auslander--Reiten quiver is simply connected as a two-dimensional simplicial complex. For a not necessarily representation finite algebra, it is easily seen that $A$ has no proper Galois covering if and only if for every bound quiver presentation $A \simeq \K Q/I$, the fundamental group of $(Q,I)$ is trivial, see \cite{MP83}.  In \cite{AS88}, the first author and Skowro\'nski call an algebra simply connected if it is triangular and has no proper Galois covering. This definition restricts to that of Bongartz and Gabriel in the representation finite case. In this terminology, Skowro\'nski's question can be reformulated to ask which algebras $A$ satisfy $\HH{1}{A}=0$ if and only if $A$ is simply connected. This statement was shown to hold true for several classes of algebras, in particular for representation finite algebras \cite{BL04}. Note also that another definition of simple connectedness, which
does not assume triangularity is used in \cite{ABL10, LeMeurTop09}. Our objective in this paper is to study this problem in case $A$ is special biserial (not necessarily triangular). Throughout, we use {\em simply connected} only for representation finite algebras, that is in the sense of Bongartz and Gabriel.  

The class of special biserial algebras was introduced by Skowro\'nski and Waschb\"{u}sch in \cite{SW83}. It has played an important r\^{o}le in the study of self-injective algebras. Special biserial algebras have a well-understood representation theory. In fact, if $A$ is special biserial, then it has a two-sided ideal $S$ such that the quotient $A/S$ is a monomial algebra, and actually a string algebra in the sense of Butler and Ringel
\cite{BR87}. In this paper, we prove the following theorem.

\begin{Thm} Let $A\simeq \K Q/I$ be a special biserial algebra. The following conditions are equivalent: 
  \begin{enumerate}[$(a)$]
 \item The fundamental group of the bound quiver $\left( Q,I\right) $ is trivial, 
     \item $A$  is representation finite and  simply connected,
     \item $\HH{1}{A}=0$,
     \item $\HH{j}{A}=0$, for every $j\geqslant 1$,
     \item $\chi(Q)= \ts{dim}_{\ts{k}} S$.
  \end{enumerate}
\end{Thm}

Thus, in particular, if $A$ satisfies the equivalent conditions of the theorem then it is necessarily triangular. Moreover, it is constrained in the sense of Bardzell and Marcos \cite{BM01}, therefore, the fundamental group of any bound quiver presentation of $A$ is trivial, or, equivalently, $A$ has no proper Galois covering.

The paper is organised as follows. In section 1, we briefly recall the necessary definitions. Sections 2 and 3 are technical: in section 2 we study the cycles involved in binomial relations for special biserial bound quivers $(Q,I)$, and in section 3 we relate the number of  projective - injective indecomposable non uniserial $\K Q/I$ modules to the Euler characteristic of $Q$ and we prove the key Proposition \ref{euler}. Section 4 studies the relations between simple connectedness, triangularity and derivations. In it we prove some lemmata used in the proof of the main Theorem, in section 5.

\label{sec:definitions}
\section{Algebras and quivers} 
Throughout this paper, $\K$ denotes an algebraically closed field and all algebras are finite dimensional associative $\K$-algebras with identity. 

Given a finite quiver $Q=(Q_0,Q_1,s,t)$, we denote by $\kQ$ its path algebra. Two paths sharing source and target are \emph{parallel}. A \emph{bypass} is a pair $(\a,p)$, where $\a$ is an arrow, and $p\not=\a$ is a path parallel to $\a$. Given two points $x,y\in Q_0$, a \emph{relation} from $x$ to $y$ is a linear combination $\sum_{i=1}^{r} \l_i w_i$, where $\l_i \in \K\backslash\{0\}$, and the  $w_i$ are distinct paths from $x$ to $y$. If  $r=1$, the relation is \emph{monomial}, whereas if $r=2$ the relation is \emph{binomial}. Let $\K Q^+$ be the two-sided ideal of $\kQ$ generated by $Q_1$. An ideal $I$ of $\kQ$ is \emph{admissible} if there exists an integer $m\geqslant 2$ such that $(\kQ^+)^m\subseteq I \subseteq (\kQ^+)^2$. In this case, the pair $(Q,I)$ is called a \emph{bound quiver}. The algebra $\kQ/I$ is basic, connected if so is $Q$, and finite-dimensional since $Q$ is finite. Given a path $u$ in $Q$, its class modulo $I$ is denoted by $\overline{u}$.

Conversely, for every finite dimensional, connected and basic algebra $A$ over an algebraically closed field $\K$, there exists a unique connected quiver $Q$ and at least one  surjective map \mor{\nu}{\kQ}{A} with $I= \ker \nu$ admissible, see \cite{BG82}. The pair $(Q, I)$ is  called a \emph{presentation} of $A$ by a \emph{bound quiver}.

For more details concerning the use of bound quivers in the representation theory of algebras we refer to \cite{ASS06}, for instance.


A bound quiver $(Q,I)$ is \emph{special biserial} \cite{SW83} if it satisfies the following conditions:
 \begin{enumerate}
  \item[$(i)$] Every point in $Q$ is the source of at most two arrows and the target of at most two arrows; 
  \item[$(ii)$] Given an arrow \mor{\a}{x}{y} in $Q$, there is at most one arrow $\b$ starting from $y$ such that $\alpha\b\not\in I$ and there is at most one arrow $\c$ arriving at $x$ such that $\c\alpha\not\in I$.              
\end{enumerate}
If, moreover, the ideal $I$ is monomial, that is, $I$ is generated by monomial relations, then $(Q,I)$ is a \emph{string bound quiver} \cite{BR87}. 

A \emph{special biserial} algebra (or a \emph{string} algebra) is an algebra admitting a presentation by a special biserial bound quiver (or a string bound quiver, respectively). If $(Q,I)$ is special biserial, then $I$ is generated by a collection of paths and a collection ${\mathcal R}$ of binomial relations (that is relations of the form $u - \l v,$ with $u,v$ parallel paths not in $I$, $\l \in \K\backslash\{0\}$) in bijection with the isomorphism classes of non-uniserial indecomposable projective-injective modules. For every such module, we arbitrarily fix a path $u$ in $Q$  such that there is a binomial relation $u - \l v$ (the relation is fixed once the path $u$ is fixed). It can be seen that a non trivial path cannot be prefix nor suffix of more than one path appearing in a binomial relation in $\R$, and that two binomial relations associated to non-isomorphic projective-injective modules have distinct starting points and distinct ending points. In the sequel, $A=\K Q/I$ denotes a special biserial algebra, unless otherwise specified. 

We recall some basic facts on special biserial algebras (see \cite{WW85, BR87}).

Let $A=\ts kQ/I$ be special biserial. Let $S$ be the socle of the direct sum of the indecomposable non uniserial projective - injective $A$ modules. Hence, as subspace of $A$ (considered as a vector space) $S$ is generated by the classes $\ol{u}$ ($=\l \ol{v}$) of paths appearing in a binomial relation $u - \l v \in \mathcal{R}$. Since $(\rad{A})\,S=S(\ts{rad}\,A)=0$,  $S$ is a two-sided ideal of $A$ and, in fact, a semisimple subbimodule of $_AA_A$. In particular, given points $x,y$ in $Q$, we have $e_xSe_y\neq 0$ if and only if there is a binomial relation $u - \l v$ from $x$ to $y$. Moreover, in this case, we have $e_xSe_y=\K \ol{u}=\K\ol{v}$. Note that $\ts{dim}_{\ts k}S$ equals the cardinality  $|\R|$ of $\R$.

Given a special biserial algebra $A$, the quotient algebra $A/S$ is a string algebra. If $A=\ts kQ/I$ with $(Q,I)$ special biserial, then $A/S\simeq \ts kQ/J$ where $J$ is the two-sided monomial ideal of $\ts kQ$ generated by $I\cup\{u,v\ |\ u - \l v\ \text{a binomial relation}\}$.

\section{Cycles in binomial relations}
\label{sec:Cycles-bs}
A \emph{cycle} in $Q$ is a path $a$ from a point $x$ to itself, and it is a \emph{simple cycle} if $x$ occurs only at the beginning and at the end of $a$. Given an arrow \mor{\a}{x}{y}, its formal inverse is the arrow \mor{\a^{-1}}{y}{x}. A \emph{walk} in $Q$ is a composition of arrows and formal inverses of arrows of $Q$,  $w=\a_1^{\e_1}\a_2^{\e_2} \cdots \a_r^{\e_r}$,  $\e_i\in\{\pm 1\}$, such that $s(\a_i^{\e_i})=t(a_{i-1}^{\e_{i-1}})$. The walk $w$  is \emph{reduced} if  contains no subwalk of the form $\a \a^{-1}$ or $\a^{-1} \a$, with $\a \in Q_1$ and it is \emph{closed} if the source of the first arrow coincides with the target of the last one. A closed walk is \emph{simple} if each point appears at most once in it, except of course its starting and ending point. In the sequel by closed walk we mean closed reduced walk.

The conditions in the definition of a special biserial bound quiver impose that binomial relations involving oriented cycles are very particular, and can be described precisely, as follows:

\begin{lem}\label{subsec:cycles-in-bsquiver}
     Let $u-\l v$ be a binomial relation from $x$ to $y$. If $u$ starts with a simple cycle $a=\a_1\cdots \a_r$ at $x$ then:  
\begin{enumerate}[$(a)$]
     \item If $x\not=y$, then
      \begin{enumerate}[$(1)$]
           \item If $u$ and $v$ share at least an arrow, then there exists a path $p$ from $x$ to $y$ and a cycle $b$ at $y$ such that $u=ap$ and $v=pb$,
	   \item If $u$ and $v$ have no arrow in common, then there exists a decomposition $a=a_1 a_2$ with $a_1$, $a_2$ non trivial paths such that $u=a^n a_1$, for some natural number $n\geqslant 1$.
      \end{enumerate}
     \item If $x=y$, then there exists a simple cycle $b=\b_1 \cdots \b_s$ at $x$ such that the relations involving the simple cycles $a$ and $b$, and the arrows incident to $x$ are of one of the following forms:
      \begin{enumerate}
            \item[$(3)$] $a ^n-\l b^m,\ \a_r \b_1,\ \b_s \a_1$, for some natural numbers $n,\ m$, and a scalar $\l\in \ts{k}\backslash\{0\}$, 
	    \item[$(4)$] $(ab)^m - \l (ba)^m,\  \a_r\a_1,\ \b_s \b_1$ for some natural number $m$, and a scalar $\l\in \ts{k}\backslash\{0\}$,
	    \item[$(5)$] $(ab)^m a - \l (ba)^m b,\ \a_r\a_1,\ \b_s \b_1$ for some natural number $m$, and a scalar $\l\in \ts{k}\backslash\{0\}$.
      \end{enumerate}
 \end{enumerate}
 
\end{lem}

\begin{proof} Let $u=ap$, with $a:\SelectTips{eu}{10}\xymatrix@1@C=15pt{x \ar[r]^{\a_1}& x_1 \ar[r]^{\a_2}& \cdots\ar[r]^{\a_r} &x}$ a simple cycle at $x$, and $p$ be a path such that $u=ap$.
\begin{enumerate}[$(a)$]
     \item Assume first that $x\not= y$, and that $u$ and $v$ share an arrow. Since $(Q,I)$ is special biserial, no path can be a common prefix (nor suffix) of two paths appearing in a binomial relation. Thus, $v$ does not start with $\a_1$. If $p$ starts with $\a_1$, then since $ap \not\in I$, the point $y$ must lie in the support of the cycle $a$, so that $p=a^n a_1$ for some $n$ and some suffix $a_1$ of $a$. But then, since $v$ does not start with $\a_1$, it cannot have a common arrow with $u$. Thus we must have that $p$ starts with the first arrow of $v$ (and $\a_r \a_1\in I$),  so that $v=pb$ with $b$ a cycle at $y$. If the cycle $b$  were stationary, we would have a relation of the form $ap - \l p$, but this would contradict the fact that $a$ is nilpotent. 

Now, if $u$ and $v$ do not share an arrow, then $p$ must start with $\a_1$, and, as before, $p=a^n a_1$ for some $n$ and some suffix $a_1$ of $a$.     

\begin{center}
\begin{tabular}{ccccc}
$\SelectTips{eu}{10}\xymatrix@R=12pt{
&\ar@(l,l)@{.}[dd]	&												&		&\ar@(r,r)@{.}[dd]  \\
& 				&\ar@{~>}[r]^p x \ar@/_/[ul]_{\a_1}	 &\ar@/^/[ur]^{\b_1}y\\
&\ar@/_/[ur]_{\a_r}	&											&		&\ar@/^/[ul]^{\b_s} \\ }$ &&\mbox{ 
 \ \ \ \ \ \ \ \ \ \ }&& $\SelectTips{eu}{10}\xymatrix@R=12pt@C=55pt{ &\\
x\ar@{~>}[r]^{a_1} \ar@{~<}`d[r] `[r]_v [r] & y\ar@{~>}`u[l] `[l]_{a_2}  [l] \\
  &}$ \\
&&&&\\
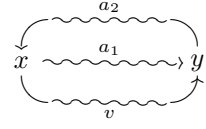
$\text{ case\ } (1)\colon a p- \l p b,\  \a_r \a_1,\ \b_s \b_1\in I$ & &&&$\text{ case\ } (2)\colon a^n a_1 - \l v \in I$\\
&&&&
\end{tabular}
\end{center}
\item Assume that $a=\a_1\cdots \a_r$ and $b=\b_1\cdots \b_s$ are simple cycles at $x$, such that $u$ begins with $a$, and $v$ begins with $b$. Note that the cycles $a$ and $b$ have no arrow in common. We distinguish two cases, according to whether $\a_r \a_1$ belongs to $I$ or not.

\begin{center}
\begin{tabular}{c}
$\SelectTips{eu}{10}\xymatrix@R=7pt{
  &\ar@(l,l)@{.}[dd]	&							&\ar@(r,r)@{.}[dd]  \\
  & 				&x \ar@/_/[ul]_ {\a_1}	 \ar@/^/[ur]^{\b_1}&\\
  &\ar@/_/[ur]_{\a_r}	&							&\ar@/^/[ul]^{\b_n} 
 }$\\
  cases $(3) - (4) - (5)$    
\end{tabular} 
     
\end{center}

Assume first that $\a_r \a_1 \not\in I$. Then  we must have $\a_r \b_1\in I$, so that $u=a^m$ for some natural number $m$. Using the same argument, we see that $v=b^n$, and thus the relation is of the form $a^m - \l b^n$. Of course, we must also have $\b_s \b_1 \not\in I$, whereas $\b_s \a_1 \in I$.

\medskip
Assume now that $a_r \a_1 \in I$. Then $\a_r \b_1 \not\in I,\ \b_s\b_1 \in I$, and $\b_s \a_1\not\in I$. Thus, there are natural numbers $n$, $m$ such that either $u=(ab)^n$ and $v=(ba)^m$, or $u=(ab)^n b$ and $v=(ba)^nb$. Assume the first case. If $m>n$ then there exists $\epsilon\geqslant 0$ such that $m-1 = n+\epsilon$ we would have $(ba)^m = b(ab)^{m-1}a = b(ab)^{n+\epsilon}a = b(ab)^n (ab)^\epsilon a =\l b (ba)^m (ab)^\epsilon a \in I$, a contradiction. Thus $m\leqslant n$, and using the same argument we obtain that in fact $n=m$, and the relation is of the form $(ab)^m - \l (ba)^m$. The second case is treated in the same way.  

\end{enumerate}

\end{proof}

Given a cycle $a = \a_1\cdots \a_r$, define $\sigma(a)$ to be the cycle $\a_2 \a_3\cdots \a_r \a_1$. A \emph{cyclic permutation} of $a$ is a cycle of the form $\sigma^j(a)$, and $\tilde{a}$ will denote the set of all cyclic permutations of $a$. Note that two simple cycles not in $I$ have a common arrow if and only if they are equal up to cyclic permutation. The following remarks will be useful.

\begin{rems}\label{thm:remarques-cycles}\ 

\begin{enumerate}[$(a)$]

   \item Note that in case \ref{subsec:cycles-in-bsquiver}, $(1)$ above, there can be a path $q$ from $y$ to a point $z$ (which might be $x$) at which there is a cycle  $c$ (which must be $a$ in case $z=x$), and a minimal relation of the form $b q - \l q c$.  However, since any cyclic permutation of $a$ and $b$ belongs to $I$, no other cycle of $\tilde{a}$ nor $\tilde{b}$ can be involved in any other binomial relation.
    
  \item If there is a relation of types \ref{subsec:cycles-in-bsquiver}  $(4)$ or $(5)$ involving two cycles $a$ and $b$, then  any cyclic permutation of these cycles belongs to the ideal $I$. Therefore, there cannot exist another binomial relation involving cycles of $\tilde{a}$ and $\tilde{b}$.

  \item  In case \ref{subsec:cycles-in-bsquiver}, $(2)$, the path $v$ may contain a subpath that is a non trivial cycle, say  $b$.   There can be relations of type $(2)$ involving other cycles of  $\tilde{a}$ or $\tilde{b}$. Moreover, if $v$ starts with a cycle, then it also ends with a cycle $b$ at $y$ and there exists a decomposition $b=b_1b_2$ such that $v=b_1 b^m$ for some natural number $m$.  A cyclic permutation of $b$ may itself be involved in another minimal relation, necessarily of type $(2)$ (see example \ref{thm:ex-cycles}, $(b)$).
 
  \item In case \ref{subsec:cycles-in-bsquiver},  $(3)$, cyclic permutations of $a$ or $b$ can be involved in several other minimal relations, either between them, or involving another cycle (see example \ref{thm:ex-cycles}, $(a)$). All such relations must be of type $(3)$.

\end{enumerate}
\end{rems}

\begin{ex}\label{thm:ex-cycles}
Consider the following quiver: $$\SelectTips{eu}{10}
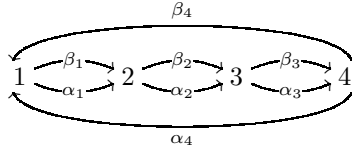
\xymatrix@C=30pt@R=15pt{1 \ar@/^/[r]|{\b_1} \ar@/_/[r]|{\a_1} & 2\ar@/^/[r]|{\b_2}\ar@/_/[r]|{\a_2} &3\ar@/^/[r]|{\b_3} \ar@/_/[r]|{\a_3}& 4\ar@(ur,ul)[lll]_{\b_4} \ar@(dr,dl)[lll]^{\a_4} }$$
Define $a=\a_1\a_2\a_3\a_4$ and  $b=\b_1 \b_2 \b_3 \b_4$.
\begin{enumerate}[$(a)$]
 \item  Let $I$ be the ideal generated by $\s^i(a)-\s^i(b),\ \a_i\b_{i+1}, \b_i\a_{i+1}$, where $1\leqslant i \leqslant 4$, and indices are to be read modulo $4$. In this case the supports of the cycles $a$ and $b$ are involved in $4$ minimal relations of type $(3)$.
\item Let $I_1$ be the ideal generated by $\s^i(a)\a_{i+1} - \s^i(b)\b_{i+1},\  \a_i\b_{i+1}, \b_i\a_{i+1}$, where $1\leqslant i \leqslant 4$ where, and again indices are to be read modulo $4$. In this case, the cycles $a$ and $b$ are involved in $4$ relations of type $(2)$.
\end{enumerate}

\end{ex}

If follows from the preceding lemma that the set $\mathcal{R}$ can  be partitioned as $\mathcal{R} = \mathcal{R}_1 \coprod \mathcal{R}_2$, where $\mathcal{R}_1$ is the set of binomial relations $u- \l v$ such that one of the paths starts or ends with a cycle, and $\R_2 = \R \backslash \R_1$. 

The following section is technical. We establish the key result \ref{euler}.

\section{The cycle graph of $(Q,I)$}  Following \cite{Hap89}, given a quiver $Q$, with $N$ connected components, the Euler characteristic of $Q$ is $\chi(Q) = |Q_1| - |Q_0| + N$. This number equals the rank of the first homology group $\SH{Q}$ of the underlying graph of $Q$, which is free abelian.  In order to compare $\ts{dim}_{\ts k}S$ and $\chi(Q)$ we  introduce an auxiliary graph $\G$, defined as follows:
\begin{enumerate}
     \item[-] $\G_0$ is the set of of simple oriented cycles in $Q$, considered up to cyclic permutation, which are prefix or suffix of a path appearing in a binomial relation in $\R_1$.
     \item[-] Given two points $\tilde{a}$ and $\tilde{b}$ in $\G$, we attach edges between them according to the following rules (see \ref{subsec:cycles-in-bsquiver}):
    \begin{enumerate}
      \item[$\cdot$] For each relation $ap - \l p b$ of type $(1)$, we attach an edge \arete{p}{\tilde{a}}{\tilde{b}}. 

      \item[$\cdot$] For each relation $a^na_1 - \l b_1 b^m$ of type $(2)$ we attach an edge \arete{p'}{\tilde{a}}{\tilde{b}}.  

      \item[$\cdot$] In case the cycles $a$ and $b$ share a point $x$ and are involved in a relation of one of the forms $(3), (4),\text{ or } (5)$ we attach an edge \arete{x}{\tilde{a}}{\tilde{b}}.
    \end{enumerate}
\end{enumerate}

\begin{rems}\label{thm:rems-Gamma}\ 
 \begin{enumerate}[$(a)$]
  \item Note that we may have $\G_0=\emptyset$, but still have oriented cycles in $Q$. However, this would mean that $\R=\R_2$.
  \item By construction $|\G_1| = |\R_1|$, this will be very useful later. Also, note that $\G$ can have multiple edges, but no loops: in case there are two (classes of) cycles $a$ and $b$ in $Q$ having $n$ common points with relations of the type $(3)$ there are exactly $n$ arrows between them. For instance, the quiver $\G$ corresponding to the bound quiver given in \ref{thm:ex-cycles} is:

$$\SelectTips{eu}{10}\xymatrix{\tilde{a} \ar@{-}@/^{3mm}/[r]  \ar@{-}@/^{1mm}/[r] \ar@{-}@/_{3mm}/[r] \ar@{-}@/_{1mm}/[r] & \tilde{b}}$$

\item The set of simple oriented cycles of $Q$ (up to cyclic permutation) which are prefixes or suffixes of a path appearing in a binomial relation gives rise to a linearly independent set in $\SH{Q}$.  Thus, we have a map \mor{\f_0}{\G_0}{\SH{Q}} whose image is linearly independent.  Note however that this does not mean that we have a natural map $\R_1\to \SH{Q}$ with the same property. Indeed we may have several binomial relations involving different cycles of a class $\tilde{a}$ (see example \ref{thm:ex-cycles}). Denote by  $\mathcal{C}_0 $ the the subgroup of $\SH{Q}$ generated by  $\ts{Im}\ \f_0$.

\item In general $\G$ is disconnected. In light of \ref{thm:remarques-cycles} it has $5$ types  of edges, each given by the type of relation that gives rise to it. Moreover, the type of an edge is an invariant for all the edges in the same connected component of $\G$. The latter  can be described as follows.
\begin{enumerate}
     \item[-] If two cycles $a=\a_1\cdots \a_r$ and $b =\b_1\cdots \b_s$ are linked by a relation of type \ref{subsec:cycles-in-bsquiver}, $(4)$ or $(5)$, then since $\a_r \a_1,\ \b_s\b_1 \in I$, no cyclic permutation of $a$ or $b$ can be involved in another binomial relation. Thus, points corresponding to relations of type \ref{subsec:cycles-in-bsquiver} $(4) \text{ and } (5)$ determine a connected component of the form $\SelectTips{eu}{10}\xymatrix@1@C=15pt{\tilde{a} \ar@{-}[r]&\tilde{b}}$,
     \item[-] Each point corresponding to binomial relations of type \ref{subsec:cycles-in-bsquiver} $(1)$ gives rise to a connected component of type $\mathbb{A}$ or $\tilde{\mathbb{A}}$,
     \item[-] Edges corresponding to relations of types  \ref{subsec:cycles-in-bsquiver} $(2)$, appear in a separate component with no loops. The same holds for edges coming from relations of type \ref{subsec:cycles-in-bsquiver} $(3)$. See \ref{thm:remarques-cycles}.
\end{enumerate}
 \end{enumerate}
\end{rems}

\begin{lem}\label{thm:H-gamma}
 There exists an injective homomorphism of groups \mor{\f_1}{\SH{\G}}{\SH{Q}} whose image $\mathcal{C}_1$ satisfies $\mathcal{C}_1 \cap \mathcal{C}_0 = 0$.
\end{lem}

\begin{proof}
 First of all, note that each edge $p$ of $\G$ determines a point of $Q$, that is the starting point of the binomial relation in $Q$ giving rise to the edge $p$. We now proceed by induction on $\chi(\G)$.

Assume $\chi(\G)=1$ so that there is a simple closed walk  $w=\SelectTips{eu}{10}\xymatrix@1@C=15pt{\widetilde{c_1} \ar@{-}[r]^{p_1}&\widetilde{c_2}\ar@{-}[r]^{p_2}&\cdots \ar@{-}[r]^{p_r}&\widetilde{c_1}}$ in $\G$. 
\begin{enumerate}
 \item[--] Assume first that this cycle lies in a component of type $(1)$. That means that there are relations of type $(1)$ in $Q$: $a_1p_1 - p_1a_2,\ a_2p_2 - p_2 a_3,\ldots, a_rp_r - p_r a_1$. Set $\f_1(w)=\sum_{j=1}^r p_j  $, as element of $\SH{Q}$.
\item[--] In light of \ref{thm:rems-Gamma}, $(d)$ we can assume that $w$ lies in a component given by relations of type $(2)$ or else by relations of type  $(3)$. Let $x_i$ be the point determined by the edge  $\SelectTips{eu}{10}\xymatrix@1@C=15pt{\tilde{c_i} \ar@{-}[r]&\widetilde{c_{i+1}}}$ of $w$, see \ref{thm:remarques-cycles}, $(c)$ and $(d)$. We build a simple cycle at $x_1$ as follows. Start at $x_1$, then:
    \begin{itemize}
         \item[.] If the edge in $\G$ comes from a relation of type $(3)$, go to $x_2$ following $c_2$,
	 \item[.] If the edge in $\G$ comes from a relation $(2)$, go from $x$ to the ending point of the binomial relation following $c_1$, then go to to $x_2$ following $c_2$.
    \end{itemize} and continue in that way. The obtained cycle is not necessarily simple, but by eliminating each proper subpath which is a nontrivial cycle, we obtain a simple cycle. 
\end{enumerate}
 
The constructed cycle does not lie in $\mathcal{C}_0$ because it does not contain all the arrows of any cycle of $Q$ it passes through.
\medskip

Assume now that the statement holds true for bound quivers such that $\chi(\G)=k-1$, and let $(Q,I)$ be such that $\chi(\G)=k\geqslant 2$. Since $\chi(\G)>0$, there exists an edge $p$ in $\G$ such that the graph $\G'$ obtained from $G$ by deleting $p$ has the same number of connected components of $\G$, and hence $\chi(\G') = \chi(\G)-1$. The graph $\G'$ corresponds to the bound quiver $(Q,I')$ where $I'$ is generated by one less binomial relation than $I$. Note that the vertices of $\G$ coincide with those of $\G'$, so that $\f_0 = \f_0'$. In addition, $\SH{\G'}$ can be regarded as a subgroup of $\SH{\G}$, and by the induction hypothesis there exists an injective map \mor{\f_1'}{\SH{\G'}}{\SH{Q}} whose image $\mathcal{C}_1'$ does not intersect $\mathcal{C}_0$.

We now extend $\f_1'$ to $\SH{\G}$. There exists a cycle in $\G$ involving the edge $p$, thus, not belonging to $\SH{\G'}$. We construct a cycle in $Q$ in the same way as we did, and it only remains to show that no integer multiple of this cycle belongs to $\mathcal{C}_0+ \mathcal{C}_1'$. Using the edge $p$ in $\G$ corresponds to changing from one cycle to another in $Q$, say from $a$ to $b$, at a point $x$ which is uniquely determined by $p$. Let $\a_1$ and $\a_2$, (respectively  $\b_1,$ and $\b_2$) be the arrows of $a$ (respectively of $b$) entering and leaving $x$.

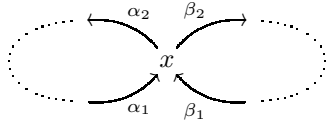
$$\SelectTips{eu}{10}\xymatrix@R=7pt{
  &\ar@(l,l)@{.}[dd]	&							&\ar@(r,r)@{.}[dd]  \\
  & 				&x \ar@/_/[ul]_ {\a_2}	 \ar@/^/[ur]^{\b_2}&\\
  &\ar@/_/[ur]_{\a_1}	&							&\ar@/^/[ul]^{\b_1} 
 }$$    
These arrows may have appeared in a previously constructed cycle of $Q$. However, since we have never before used the point $x$ to change from $a$ to $b$, each time $\a_1$ appeared in the cycles of $\mathcal{C}_0 + \mathcal{C}_1'$, the arrow $\alpha_2$ also appeared. Thus, in every element of $\mathcal{C}_0 + \mathcal{C}_1'$, the coefficient of $\a_1$ is the same as that of $\a_2$, and this is not the case for the cycle we have just constructed.
\end{proof}

\begin{prop}\label{euler}  
Let $(Q,I)$ be a special biserial bound quiver. Then $\ts{dim}_{\K} S \leqslant \chi(Q)$. Moreover, if equality holds, then $Q$ is acyclic.
\end{prop}

\begin{proof}  Let $u-\l v\in{\R}_2$ be a relation from $x$ to $y$. Note that neither $u$ nor $v$ starts with a cycle, and they have at least the point $y$ in common. Let $z$ be the first point which is common to $u$ and $v$ and different from $x$, thus $u=u' u''$ and $v=v' v''$, where $u'$ and $v'$ go from $x$ to $z$. Moreover, let $u_1$ be the path constructed from $u'$ by deleting every occurrence of a cycle, if there is one, and $v_1$ constructed in an analogous way. Since $u-\l v\in{\R}_2$ we have that $u_1 v_1^{-1}$ is a reduced non-oriented cycle of $Q$. Define \mor{\f_2}{\R_2}{\SH{Q}} by $\f_2(u-\l v)=u_1-  v_1$. The paths $u$ and $u_1$ have their first arrow in common, and since no path in $Q$ can be a prefix (nor suffix) of more than one path appearing in some binomial relation, the image of this map is a linearly independent set. Let $\mathcal{C}_2$ be subgroup of $\SH{Q}$ generated by $ \ts{Im}\ \f_2 $.

In \ref{thm:H-gamma} we constructed an injection \mor{\f_1}{\SH{\G}}{\SH{Q}}, whose image $\mathcal{C}_1$ satisfies $\mathcal{C}_1\cap \mathcal{C}_0 = 0$. By construction, any cycle $\mathcal{C}_0 +\mathcal{C}_1$ has only arrows belonging to cycles which are prefixes or sufixes of paths appearing in binomial  relations. Since the first arrow of a path of a relation from $\R_2$ cannot appear in any such cycle, we have $\mathcal{C}_2 \cap \left( \mathcal{C}_0 + \mathcal{C}_1\right) = 0$. Thus $\mathcal{C}_0 + \mathcal{C}_1 + \mathcal{C}_2 = \mathcal{C}_0 \oplus \mathcal{C}_1 \oplus \mathcal{C}_2$ is a subgroup of $\SH{Q}$, and thus

\begin{eqnarray*}
     \ts{dim}_{\K} S & =  &|\R|\\
			  & = & |\R_1|+ |\R_2|\\
			& = & |\G_1| + \ts{rk}\ \mathcal{C}_2\\
			& \leqslant  & N+ |\G_1| + \ts{rk}\ \mathcal{C}_2\\
			& = & \chi(\G) + |\G_0| + \ts{rk}\ \mathcal{C}_2\\
			& = &  \ts{rk}\ \mathcal{C}_0 + \ts{rk}\ \mathcal{C}_1 +\ts{rk}\ \mathcal{C}_2\\
			& \leqslant &\ts{rk}\ \SH{Q}\\
			 & = &\chi(Q).
\end{eqnarray*}
Assume now that $\chi(Q) = \ts{dim}_{\K} S$. Then all the inequalities  must be equalities, in particular $N=0$. Therefore  $\G$ is empty, and so is $\R_1$. Moreover,  we must have that $\mathcal{C}_2= \SH{Q}$, which is only formed by non oriented cycles.
\end{proof}

\begin{cor}\label{cor-euler} Let $(Q,I)$ be a special biserial bound quiver such that  $\ts{dim}_{\K} S = \chi(Q)$.
 If $u,v$ are distinct parallel paths, then $\ol u$ and  $\ol v$ are proportional. In particular, $Q$ has no bypasses.
\end{cor}

\begin{proof} The hypothesis $\ts{dim}_{\K} S = \chi(Q)$ implies that $Q$ has no oriented cycles and $\mathcal R=\mathcal R_2$.

Let $u,v$ be parallel paths. Consider the shortest non-trivial prefixes $u'\ v'$ of $u$ and $v$, respectively, such that $u'$ and $v'$ are parallel. Then $u'v'^{-1}$ is a simple closed walk in $Q$. Therefore, $u'$ and $v'$ are bound by a binomial relation. Moreover, if we write $u=u'u''$ and $v=v'v''$, then $u''$ is trivial if and only if $v''$ is trivial, because $Q$ has no oriented cycles. Thus, $u$ and $v$ are bound by a binomial relation or both lie in $I$ according to whether $u''$ and $v''$ are both trivial or both non trivial, respectively.        

The second statement follows directly, since $I$ is admissible.
  \end{proof}

We refer the reader to \cite{MP83} for the definition of the fundamental group of a bound quiver. See also \cite{Skow92, AdlP96, PS01} for relations with Hochschild cohomology. Recall from \cite{BM01} that an algebra $A\simeq \K Q/I$ is called \emph{constrained} if for each arrow \mor{\a}{x}{y} in $Q$ we have $\ts{dim}_{\ts{k}} e_x A e_y \leqslant 1$. It is shown in \cite{BM01} that if $A$ is constrained, then the fundamental groups of any two presentations of $A$ are isomorphic.
\begin{cor}\label{cor-euler2}
     Let $(Q,I)$ be a special biserial bound quiver with $\chi(Q)= \ts{dim}_{\ts{k}} S$. If $\K Q/I \simeq \K Q/I'$, then $\pi_1(Q,I) \simeq \pi_1(Q,I')$,
\end{cor} 

\begin{proof}
     This follows from \cite{BM01}, and the absence of bypasses in $Q$.
\end{proof}

\section{Derivations, triangularity and fundamental groups}

\begin{lem}\label{charactersvanish} Let $(Q,I)$ be a special biserial bound quiver. If $\ \ts{Hom}(\pi_1(Q,I),\ts k^+)=0$, then $\ts{dim}_{\ts k}S = \chi(Q)$. In particular, $Q$ is acyclic.
\end{lem}

\begin{proof} Recall from the description of $\pi_1(Q,I)$ given in \cite{MP83} for instance,   that this group is generated by $\chi(Q)$ elements satisfying $\ts{dim}_{\ts k}S$ relations. Therefore, $\ts{Hom}(\pi_1(Q,I),\ts k^+)$ is isomorphic, as a $\ts k$-vector space, to a subspace of $\ts k^{\chi(Q)}$ given by $\ts{dim}_{\ts k}S$ relations. Therefore, $\ts{Hom}(\pi_1(Q,I),\ts k^+)=0$ implies that $\ts{dim}_{\ts k}S\geqslant \chi(Q)$. The conclusion then follows from \ref{euler}.     
\end{proof}

\begin{prop}\label{thm:H1=0-imply-Euler} Let $A$ be a special biserial algebra such that $\ts{HH}^1(A)=0$ or else such that there exists a special biserial presentation $A\simeq \ts kQ/I$ such that  $\pi_1(Q,I)=1$, then $\ts{dim}_{\ts k}S=\chi(Q)$. In particular, $Q$ is acyclic.
\end{prop}
\begin{proof} If $\ts{HH}^1(A)=0$, then it follows from \cite[3.2]{AdlP96} that $\ts{Hom}(\pi_1(Q,I),\ts k^+)=0$. Similarly, if $\pi_1(Q,I)=1$, then $\ts{Hom}(\pi_1(Q,I),\ts k^+)=0$. We conclude using
\ref{charactersvanish}.\end{proof}

Recall that string algebras are tame, and have two kinds of indecomposable modules,  the so-called \emph{band} modules and the \emph{string} modules. Moreover, with our notations every $A$-module is either projective-injective or an $A/S$-module. See \cite{WW85, BR87}, for more details.

\begin{lem}\label{thm:rep-finite}
  If $\chi(Q)=\ts{dim}_{\ts k}S$, then the string algebra $A/S$ has no band module. Therefore, $A$ and $A/S$ are representation finite.
\end{lem}
\begin{proof}
Let $A\simeq\ts kQ/I$ be a special biserial presentation and let $J$ be the ideal of $\ts kQ$ generated by $I$ and the paths in $Q$ appearing in a binomial relation in $I$. Hence, $A/S\simeq \ts kQ/J$. It follows from \ref{cor-euler} that every simple cycle in $Q$ contains a path or a formal inverse of a path in $Q$ that lies in $J$. Thus, there is no band in $(Q,J)$ and, therefore, no band module over $A/S$. This shows that $A/S$ and, therefore, $A$ are both  representation finite.  
\end{proof}

We now need to give a precise description of the fundamental group $\pi_1(Q,I)$, in case $(Q,I)$ is a triangular special biserial bound quiver. In order to do so, we recall some terminology and results from \cite{AdlP96, CdlP03}.

Following \cite{CdlP03}, given a presentation \mor{\nu}{\K Q}{A}, the algebra $A$ is said to be \emph{of the first kind} with respect to $\nu$ if for every point $x$ and associated indecomposable projective $A$ module $P_x$, every indecomposable summand of  $\rad{P_x}$ is of the first kind with respect to the universal Galois covering  associated to $\nu$. The main result of \cite{CdlP03} states that if $A$ is a triangular algebra of the first kind with respect to a presentation \mor{\nu}{\K Q}{A}, then the fundamental group $\pi_1(Q,\ker{\nu})$ is free.

Now, following \cite{AdlP96}, let $(Q,I)$ be a bound quiver with $Q$ acyclic, \mor{\nu}{\K Q}{A} a presentation of $A$, with kernel $I$, and $x$ be a source in $Q$. Let $x^+$ be the set of arrows starting at $x$, and let $\approx$ be the smallest equivalence relation on $x^+$ such that $\a \approx \b$ whenever there exist $y\in Q_0$ and a minimal relation $\sum_{i=1}^r \l_i w_i \in I$, from $x$ to $y$, such that $w_1=\a w'_1$ and $w_2=\b w'_2$. Further, denote by $t_x(\nu)=t(\nu)$ the number of equivalence classes of $\approx$. Let $Q'$ be the quiver obtained from $Q$ by deleting $x$, $I'=I\cap \K Q'$, and $A'=\K Q'/I'$, so that $A$ is a one-point extension of $A'$. Then  \cite[2.2]{CdlP03} asserts that $\pi_1(Q,I)$ is the free product of the fundamental groups of the connected components of$(Q',I')$ and the free group in $t(\nu) -1$ generators.

\begin{lem}\label{lem:pifree}
     Let $(Q,I)$ be a triangular special biserial bound quiver. Then $\pi_1(Q,I)$ is free of rank $\chi(Q) - \ts{dim}_{\ts{k}} S$.
\end{lem}
\begin{proof} First, note that if \mor{p}{(\tilde{Q}, \tilde{I})}{(Q,I)} is a Galois covering, then $(\tilde Q, \tilde I)$ is also special biserial. In addition, if $w$ is a string (or a band) in $(Q,I)$, then there exists a string $\tilde{w}$ in $(\tilde{Q}, \tilde{I})$ such that $p(\tilde{w}) = w$. Furthermore, if $M(w)$ denotes the string  module corresponding to $w$, then $M(w) = p_\l M(\tilde{w})$, where $p_\l$ denotes the push-down functor associated to $p$ (see \cite{BG83}). Because indecomposable projectives are string modules, this implies that special biserial algebras are of the first kind. In light of the main Theorem of \cite{CdlP03} cited above, this shows that $\pi_1(Q,I)$ is free.

Now, if $x$ is a source in $Q$, then $t(\nu)$ is $1$ or $2$, according to whether the projective $P_x$ is also injective or not. The result then follows by induction on the number of points in $Q$, using  \cite[2.2]{CdlP03}.

\end{proof}

The following example shows that the statement does not hold true if one drops the triangularity hypothesis.

\begin{ex}
  Let $A$ be the algebra given by the quiver
  \begin{equation}
    \SelectTips{eu}{10}
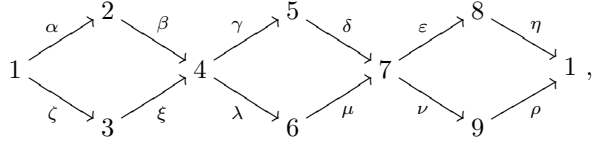
\xymatrix@R=2ex{
& 2 \ar[rd]^{\b} & & 5 \ar[rd]^{\d} & & 8 \ar[rd]^{\eta}& 
\\
1 \ar[ru]^{\alpha} \ar[rd]_{\zeta} & & 4 \ar[ru]^{\c}
\ar[rd]_{\lambda} & & 7 \ar[ru]^{\varepsilon} \ar[rd]_{\nu} & & 1\ ,
\\
& 3 \ar[ru]_{\xi} & & 6 \ar[ru]_{\mu} & & 9 \ar[ru]_{\rho}
}\notag
  \end{equation}
where the two copies of the point $1$ are identified, bound by the
relations
\begin{equation}
  \b\lambda=\xi\c=\d\nu=\mu\varepsilon=\eta\zeta=\rho\alpha=0,\
  \alpha\b\c\d=\zeta\xi\lambda\mu,\
  \c\d\varepsilon\eta=\lambda\mu\nu\rho,\
  \varepsilon\eta\alpha\b=\nu\rho\zeta \xi\ .\notag
\end{equation}
Then $A$ is a non-triangular special biserial algebra and $\pi_1(Q,I)\simeq \mathbb{Z}\times\mathbb{Z}$.
\end{ex}

\section{The main result}  We recall that a representation-finite algebra $A$ is called simply connected if its Auslander - Reiten quiver is simply connected as a two-dimensional simplicial complex, see \cite{BG82}. This is equivalent to saying that $A$ is triangular and has no proper Galois covering, see \cite{AS88}. As promised, we now establish a relationship between the simple connectedness of a special biserial algebra, the vanishing of its Hochschild cohomology groups and the dimension of $S$. Note that, by definition, simple connectedness implies representation - finiteness.
\begin{thm}\label{thm:MainThm}
     Let $A=\K Q/I$ be a special biserial algebra. Then, the following conditions are equivalent:
\begin{enumerate}[$(a)$]
     \item $\pi_1(Q,I) = 1$, 
     \item $A$ simply connected,
     \item $\HH{1}{A}=0$,
     \item $\HH{j}{A}=0$, for every $j\geqslant 1$,
     \item $\chi(Q)= \ts{dim}_{\ts{k}} S$.
\end{enumerate}
\end{thm}
\begin{proof}\ 

$(a)$ implies $(b)$ :  If $\pi_1(Q,I)=1$ then \ref{thm:H1=0-imply-Euler} implies that $Q$ is acyclic, and that $\chi(Q)= \ts{dim}_{\ts{k}} S$. Then \ref{cor-euler2}  implies that the fundamental group of every presentation of $A$ is trivial. Finally, \ref{thm:rep-finite} gives the remaining part.

$(b)$ implies $(c)$ : If $\pi_1(Q,I)=1$ then \ref{thm:H1=0-imply-Euler} implies that $Q$ is acyclic, and, since $A$ is representation finite, Theorem (4.3) in \cite{BL04} gives $\HH{1}{A}=0$.

$(c)$ and $(d)$ are equivalent : If $\HH{1}{A}=0$, then by \ref{thm:rep-finite} the algebra $A$ is representation finite. Corollary 4.4 in \cite{BL04} then yields that $\HH{i}{A}=0$ for $i\geqslant 2$. It is trivial that $(d)$ implies $(c)$.

$(d)$ implies  $(e)$ : Since $\HH{1}{A}=0$, then \ref{thm:H1=0-imply-Euler} implies that $Q$ is acyclic and $\chi(Q)= \ts{dim}_{\ts{k}} S$. The result follows from \ref{lem:pifree}.

$(d)$ implies  $(e)$ : This follows from  \ref{euler} and  \ref{lem:pifree}.

\end{proof}

Notice that because of \ref{cor-euler2}, and the fact that an algebra $A$ has no proper Galois covering if and only if the fundamental group of any presentation is trivial, it follows from \ref{thm:MainThm} $(a)$ that the conditions of the theorem are further equivalent to saying that $A$ is representation-finite and has no proper Galois covering.
 
We now derive some consequences, the first of which deals with the Lie algebra structure of $\HH{1}{A}$.

\begin{thm}\label{thm:SES-HH1}
  Let $A=\K Q/I$ be a special biserial algebra without bypass. Then there is a short exact sequence of $\ts k$-vector spaces
$$\SelectTips{eu}{10}\xymatrix@C=12pt{0\ar[r]& \HH{1}{A}\ar[r]^{p^*}& \HH{1}{A/S} \ar[r]& \K^{\ts{dim}_{\ts
        k}S}\ar[r]& 0 } $$
where the map $p^*$ is a morphism of Lie algebras.\end{thm}

\begin{proof}
     Recall that $A/S \simeq \K Q/J$, and that $J$ is a monomial ideal, so that $\pi_1(Q,J)$ is the free group on $\chi(Q)$ generators. Thus, there is a natural surjective group homomorphism \mor{p}{\pi_1(Q,J)}{\pi_1(Q,I)} obtained by factoring out the binomial relations. Since $Q$ has no bypasses, we have, from \cite{PS01}, that $\HH{1}{A} \simeq \Hom{}{\pi_1(Q,I)}{\K^+}$ and $\HH{1}{A/S} \simeq \Hom{}{\pi_1(Q,J)}{\K^+}$. Moreover, the derivations of $A$ and those of $A/S$ are diagonalisable, see \cite{AdlP96, PS01}, that is, for every arrow $\a$, and every derivation $\partial$ of $A$ or $A/S$, the image of $\overline{\a}$ under $\partial$ is a scalar multiple of itself. The morphism $p^*$ is obtained from $p$ upon applying the functor $\Hom{}{-}{\K^+}$. Finally, from \ref{lem:pifree}, we obtain $\ts{dim}_{\K} \HH{1}{A/S}=\chi(Q)$, and $\ts{dim}_{\K} \HH{1}{A}=\chi(Q)-\ts{dim}_{\K} S$ and the dimension of the cokernel of $p^*$ follows. Finally, since the derivations are diagonalisable, the Lie algebras $\HH{1}{A/S}$ and $\HH{1}{A}$  are abelian, and the map $p^*$ is thus trivially a Lie algebra homomorphism. 
\end{proof}

We refer to \cite{Skow92}, for instance, for the definition of separated, coseparated, or strongly simply connected algebras. We now easily  deduce conditions equivalent to those of our theorem \ref{thm:MainThm} for special biserial representation - finite algebras.

\begin{cor} Let $A$ be a special biserial algebra. The following conditions are equivalent:
\begin{enumerate}[$(a)$]
     \item $A$  simply connected (thus, by definition, representation-finite),
     \item There exists a string bound quiver presentation $(Q,I')$ of  $A/S$ such that $\pi_1(Q,I')$ is the free group of rank ${\ts{dim}_{\ts k}S}$,
     \item $Q$ has no bypass and $\ts{HH}^1(A/S)\simeq\ts k^{\ts{dim}_{\ts k}S}$.
\end{enumerate}
\end{cor}

\begin{proof}\ 

  $(a)$ implies $(c)$ : From  \ref{thm:MainThm}, if $A$ is simply connected, then $\HH{1}{A}=0$, and then from \ref{euler} and \ref{cor-euler}, $Q$ has no bypasses. Finally, the exact sequence of  \ref{thm:SES-HH1} gives the result.
  
  $(c)$ implies $(b)$ :  The exact sequence of \ref{thm:SES-HH1} gives $\HH{1}{A}=0$, and then  \ref{thm:MainThm} gives that ${\ts{dim}_{\ts k}S}=\chi(Q)$. On the other hand, since $I'$ is monomial, the group $\pi_1(Q,I')$ is free in $\chi(Q)$ generators.
  
  $(b)$ implies $(a)$ : The hypothesis implies that ${\ts{dim}_{\ts k}S}=\chi(Q)$, and the result follows from  \ref{thm:MainThm}.
     
\end{proof}

\begin{rems}\label{thm:rem-sb-rep-fin}\ 
\begin{enumerate}[$(1)$]
             \item We have further equivalent conditions, namely if $A$ is special biserial and representation - finite, then the following are equivalent:
		  \begin{enumerate}[(a)]
		      \item $A$ is simply connected,
		      \item $A$ is separated,
		      \item $A$ is co-separated,
		      \item $A$ is strongly simply connected. 
		  \end{enumerate} 
		  Indeed, because of  \cite[2.3, 4.1]{Skow92} condition $(d)$ implies $(b)$  and $(c)$, which imply $(a)$. Finally, $(a)$ implies $(d)$ follows from \cite{BG83}.
	\item Let $A$ be a simply connected triangular special biserial algebra. Then it follows from \ref{thm:rem-sb-rep-fin}, (a),  and \cite{B85}  that the Auslander-Reiten quiver of $A$ admits both a postprojective and a preinjective component.
	\item  Let $A$ be a simply connected special biserial algebra. By \cite{AL98}, there exists a poset $\Sigma$ and an a ideal $J$ of the incidence algebra $\ts k\Sigma $ which is generated by classes of paths in the quiver of $\Sigma$,  such that $A\simeq \ts k\Sigma/J$. In particular, $A$ is schurian.	
	\item In \cite{ACMT07} are given criteria for the strong simple connectedness of quotients of incidence algebras.
\end{enumerate}
\end{rems}

Recall that for schurian algebras, or, more generally for algebras $A$  having a semi-normed basis, the simplicial homology $\ts{SH}_*(A)$ and cohomology  groups $\ts{SH}^*(A;\ts k^+)$ of $A$ (with coefficients in $\ts k^+$) are defined, see \cite[2.1]{BG83} and \cite{MdlP99}. Moreover, following \cite{Bus04}, these groups have a clear interpretation as the homology or cohomology groups of a CW-complex. 

\begin{cor}
\label{cor:schurian-sc}
  Let $A$ be a schurian special biserial algebra. The following are equivalent:
\begin{enumerate}[$(a)$]
     \item $\HH{1}{A}=0$, 
     \item $\ts{SH}_1(A)=0$,
     \item $\ts{SH}^1(A;\ts k^+)=0$.
\end{enumerate}
\end{cor}

\begin{proof}
     Since $A$ is schurian, its quiver has no bypasses. From the previous results, if $\HH{1}{A}=0$, then $\pi_1(Q,I)$ is trivial, and hence so is its abelianisation,  $\ts{SH}_1(A)$. Then we have, for every presentation $(Q,I)$ of $A$:
     \begin{equation}
    \begin{array}{rcl}
      \ts{HH}^1(A) & \simeq & \ts{Hom}(\pi_1(Q,I),\ts k^+)\\
&\simeq & \ts{Hom}(\ts{SH}_1(A),\ts k^+)\\
&\simeq & \ts{SH}^1(A;\ts k^+)
    \end{array}\notag
  \end{equation}
where the first isomorphism comes from \cite{PS01}, the second is the Hurewicz Theorem (see \cite[4.29]{ROT88}, and the third is given by the Dual Universal Coefficients Theorem (see \cite[12.11]{ROT88}). 
\end{proof}

\section*{Acknowledgements} The first  author gratefully acknowledges partial support from NSERC of Canada and the Universit\'e de Sherbrooke. The second author acknowledges partial support from the Universidad San Francisco de Quito and financial support from Universit\'e de Sherbrooke, as well as the Sherbrooke group for their hospitality. The third author acknowledges financial support from Universit\'e de Sherbrooke and from the CMLA, ENS de Cachan. He also wishes to thank the first two authors for their warm hospitality during visits at Universit\'e de Sherbrooke.  The authors thank Jean-Philippe Morin for useful discussions.

\bibliography{HH1-sb-2011}

\begin{thebibliography}{10}

\bibitem{AdlP96}
I.~Assem and J.~A. de~la Pe{\~n}a.
\newblock The fundamental groups of a triangular algebra.
\newblock {\em Comm. Algebra}, 24(1):187--208, 1996.

\bibitem{AS88}
I.~Assem and A.~Skowro{\'n}ski.
\newblock On some classes of simply connected algebras.
\newblock {\em Proc. London Math. Soc.}, 56(3):417--450, 1988.

\bibitem{ABL10}
I.~Assem, J.~C.~Bustamante, and P.~Le~Meur.
\newblock Coverings of laura algebras: the standard case.
\newblock {\em J. Algebra}, 323(1):83--120, 2010.

\bibitem{ACMT07}
I.~Assem, D.~Castonguay, E.~N.~Marcos, and S.~Trepode.
\newblock Quotients of incidence algebras and the {E}uler characteristic.
\newblock {\em Comm. Algebra}, 35(4):1075--1086, 2007.

\bibitem{AL98}
I.~Assem and S.~Liu.
\newblock Strongly simply connected algebras.
\newblock {\em J. Algebra}, 207(2):449--477, 1998.

\bibitem{ASS06}
I.~Assem, D.~Simson, and A.~Skowro{\'n}ski.
\newblock {\em Elements of the representation theory of associative algebras.
  {V}ol. 1}, volume~65 of {\em London Mathematical Society Student Texts}.
\newblock Cambridge University Press, Cambridge, 2006.
\newblock Techniques of representation theory.

\bibitem{BM01}
M.~J.~Bardzell and E.~N.~Marcos.
\newblock {$H^1$} and presentations of finite dimensional algebras.
\newblock In {\em Representations of algebras ({S}\~ao {P}aulo, 1999)}, volume
  224 of {\em Lecture Notes in Pure and Appl. Math.}, pages 31--38, New York,
  2002. Dekker.

\bibitem{BG82}
K.~Bongartz and P.~Gabriel.
\newblock Covering spaces in representation-theory.
\newblock {\em Invent. Math.}, 65(3):331--378, 1981/82.

\bibitem{B85}
K.~Bongartz.
\newblock Quadratic forms and finite representation type.
\newblock 1146:325--339, 1985.

\bibitem{BG83}
O.~Bretscher and P.~Gabriel.
\newblock The standard form of a representation-finite algebra.
\newblock {\em Bull. Soc. Math. France}, 111(1):21--40, 1983.

\bibitem{BL04}
R.~O.~Buchweitz and S.~Liu.
\newblock Hochschild cohomology and representation-finite algebras.
\newblock {\em Proc. London Math. Soc. (3)}, 88(2):355--380, 2004.

\bibitem{Bus04}
J.~C.~Bustamante.
\newblock The classifying space of a bound quiver.
\newblock {\em J. Algebra}, 277(2):431--455, 2004.

\bibitem{BR87}
M.~C.~R. Butler and C.~M.~Ringel.
\newblock Auslander-{R}eiten sequences with few middle terms and applications
  to string algebras.
\newblock {\em Comm. Algebra}, 15(1-2):145--179, 1987.

\bibitem{CdlP03}
D.~Castonguay and J.~A. de~la Pe{\~n}a.
\newblock On the inductive construction of {G}alois coverings of algebras.
\newblock {\em J. Algebra}, 263(1):59--74, 2003.

\bibitem{PS01}
J.~A.~de~la Pe{\~n}a and M.~Saor{\'{\i}}n.
\newblock On the first {H}ochschild cohomology group of an algebra.
\newblock {\em Manuscripta Math.}, 104(4):431--442, 2001.

\bibitem{Hap89}
D.~Happel.
\newblock {\em Hochschild cohomology of finite-dimensional algebras}, volume
  1404 of {\em Lecture Notes in Math.}, pages 108--126.
\newblock Springer, Berlin, 1989.

\bibitem{LeMeurTop09}
P.~Le~Meur.
\newblock Topological invariants of piecewise hereditary algebras.
\newblock {\em Trans. Amer. Math. Soc.}, 363(4):2143--2170, 2011.

\bibitem{MP83}
R.~Mart{\'{\i}}nez-Villa and J.~A. de~la Pe{\~n}a.
\newblock The universal cover of a quiver with relations.
\newblock {\em J. Pure Appl. Algebra}, 30(3):277--292, 1983.

\bibitem{MdlP99}
M.~I.~R. Martins and J.~A. de~la Pe{\~n}a.
\newblock Comparing the simplicial and the {H}ochschild cohomologies of a
  finite-dimensional algebra.
\newblock {\em J. Pure Appl. Algebra}, 138(1):45--58, 1999.

\bibitem{ROT88}
J.~J.~Rotman.
\newblock {\em An introduction to algebraic topology}, volume 119 of {\em
  Graduate Texts in Mathematics}.
\newblock Springer-Verlag, New York, 1988.

\bibitem{Skow92}
A.~Skowro{\'n}ski.
\newblock {S}imply connected algebras and {H}ochschild cohomologies.
\newblock In {\em {P}roceedings of the sixth international conference on
  representation of algebras}, number~14 in {O}ttawa-{C}arleton Math. Lecture
  Notes Ser., pages 431--448, Ottawa, ON, 1992.

\bibitem{SW83}
A.~Skowro{\'n}ski and J.~Waschb{\"u}sch.
\newblock Representation-finite biserial algebras.
\newblock {\em J. Reine Angew. Math.}, 345:172--181, 1983.

\bibitem{WW85}
B.~Wald and J.~Waschb{\"u}sch.
\newblock Tame biserial algebras.
\newblock {\em J. Algebra}, 95(2):480--500, 1985.

\end{thebibliography}

\bibliographystyle{plain}
\end{document}